\documentclass[12pt]{amsart}
\usepackage{amsmath,amsfonts,amsthm,amssymb,mathrsfs}
\usepackage{times, graphicx}
\usepackage{accents, ulem, color, comment}
\theoremstyle{plain}
\newtheorem{theorem}{Theorem}[section]
\newtheorem{lemma}[theorem]{Lemma}
\newtheorem{lem}[theorem]{Lemma}
\newtheorem{prop}[theorem]{Proposition}

\theoremstyle{definition}
\newtheorem{definition}[theorem]{Definition}
\newtheorem{defn}[theorem]{Definition}
\theoremstyle{remark}

\newtheorem*{rmk}{Remark}

\newcommand{\vphi}{\varphi}

\newcommand{\pl}{\partial}

\newcommand{\na}{\nabla}
\newcommand{\lt}{\left}
\newcommand{\rt}{\right}
\newcommand{\rw}{\rightarrow}
\newcommand{\R}{\mathbb{R}}

\renewcommand{\tilde}{\widetilde}
\DeclareMathOperator{\tr}{tr}

\makeatletter
\newcommand*\bigcdot{\mathpalette\bigcdot@{.5}}
\newcommand*\bigcdot@[2]{\mathbin{\vcenter{\hbox{\scalebox{#2}{$\m@th#1\bullet$}}}}}
\makeatother

\title{The mass of the static extension of small spheres}
\author[B. Harvie]{Brian Harvie}
\address{National Center for Theoretical Sciences\\No. 1 Sec. 4 Roosevelt Rd., National Taiwan
University\\Taipei,106, Taiwan}
\email{bharvie@ncts.ntu.edu.tw}
\author[Y.-K. Wang]{Ye-Kai Wang}
\address{Department of Applied Mathematics, National Yang Ming Chiao Tung University,
Hsinchu, Taiwan}
\address{National Center for Theoretical Sciences\\No. 1 Sec. 4 Roosevelt Rd., National Taiwan
University\\Taipei,106, Taiwan}
\email{ykwang@math.nycu.edu.tw}
\thanks{Y.-K. Wang is supported by Taiwan NSTC grant
109-2628-M-006-001-MY3. We would like to thank professor Pengzi Miao and Divid Wiygul for their interest in the paper.}
\date{}
\begin{document}
\maketitle
\begin{abstract}
We give a simple proof to the computation of ADM mass of the static extensions of small spheres in Wiygul \cite{W1, W2}. It makes use of the mass formula $m = \frac{1}{4\pi} \int_{\pl M} \frac{\pl V}{\pl \nu}$ for an asymptotically flat static manifold with boundary.
\end{abstract}
\section{Introduction}
Let $\gamma$ be a Riemannian metric on $\Omega = \{ |x| \le 1\} \subset \R^3$. Denote the induced metric and mean curvature of $\pl\Omega = S^2$ by $\sigma$ and $H$ and henceforth $(\sigma, H)$ is referred to as the {\it Bartnik data} of $S^2$. The Bartnik quasi-local mass\footnote{We make several simplifications in the definition given below. For the complete definition and various modifications, see Anderson \cite{A2}.} was introduced in \cite{B} to measure the gravitational energy\footnote{In general relativity, $(\Omega, \gamma)$ corresponds to a region in a time-symmetric initial data set and hence has nonnegative scalar curvature. We do not make this assumption here.} contained in $(\Omega, \gamma)$. Let $\mathcal{PM}$ denote the set of complete Riemannian metrics on $M= \{ |x|\ge 1\} \subset \R^3$ that have nonnegative scalar curvature and are asymptotically flat in the sense that
\begin{align*}
|g_{ij}(x) - \delta_{ij}| + |x||g_{ij,k}(x)| + |x|^2 |g_{ij,kl}(x)| \le O(|x|^{-1}).
\end{align*}
\begin{defn}
We define the {\it extensions} of $(\Omega,\gamma)$, denoted by $\mathcal{PM}[\Omega, \gamma]$, to consist of metrics $g \in \mathcal{PM}$ such that
\begin{align}
&\mbox{the induced metric on } \pl M = S^2 \mbox{ is }  \sigma, \\
 & \mbox{the mean curvature of } \pl M \mbox{ is equal to } H, \\
 & \mbox{and } M \mbox{ has no stable minimal two-sphere enclosing } \pl M. \label{no horizon}
\end{align}
Our convention is that for unit sphere in Euclidean space, $S^2(1) \subset \bar B_1$ and $S^2(1) \subset \R^3 \setminus B_1$, both mean curvatures are equal to $2$.

\end{defn}

\begin{defn}
The Bartnik mass of $(\Omega, \gamma)$ is taken to be the infimum of the ADM mass of the extensions
\[m_B(\Omega, \gamma) = \inf \{ m_{ADM}(g) \,| \, g \in \mathcal{PM}[\Omega,\gamma] \} \]
\end{defn}

\begin{comment}
Given a Riemannian metric $\sigma$ and a function $H$ on a topological sphere $S^2$, henceforth referred to as the Bartnik data, we define the class
\begin{align*}
\mathcal{PM}[\sigma,H] := \{ g \in \mathcal{PM} \,|& \mbox{ the induced metric on } \pl M = S^2 \mbox{ is }  \sigma, \\
 & \mbox{ the mean curvature of } \pl M \mbox{ is less than } H, \\
&\,\, \pl M \mbox{ is outer-minimizing}. \}
\end{align*}

\begin{definition}
The Bartnik-Bray outer mass of the Bartnik data $(\sigma, H)$ is taken to be the infimum of the ADM mass \cite{} of the admissible extensions
\begin{align*}
m_B(\sigma, H) = \inf \{ m_{ADM}(g) \,|\, g \in \mathcal{PM}[\sigma, H] \}.
\end{align*}
\end{definition}
\end{comment}

Closely associated with the Bartnik mass is the concept of {\it static extensions}. A static extension of $(\Omega,\gamma)$ (or the Bartnik data $(\sigma, H)$) consists of a  metric $g \in \mathcal{PM}[\Omega,\gamma]$ and a  function (static potential) $V$ on $M$ satisfying the static equations
\begin{align*}
\mbox{Hess}_g V = V \mbox{Ric}(g),\quad 
\Delta_g V =0.
\end{align*}
Bartnik conjectured that the infimum in the definition of the Bartnik mass is achieved by the ADM mass of a unique static extension.

Huisken-Ilmanen \cite[Theorem 9.2]{HI} showed that the large sphere limits of the Bartnik mass of asymptotically flat manifolds coincides with the ADM mass. Wiygul \cite{W1,W2} studied the small sphere limits.
\begin{theorem}[Corollary 1.18 of \cite{W2}]\label{Wiygul_main}
Let $p$ be a point in a smooth Riemannian manifold. Let $B_r$ be the closed geodesic ball of radius $r$ and center $p$. Let $\sigma_r, H_r$ be the induced metric and mean curvature of $\pl B_r$. Then for $r$ sufficiently small, there exists a static extension $(g_r, V_r)$, close to $(g_{euc}, 1)$, of $B_r$. Moreover, the ADM mass has Taylor expansion
\begin{align}
m_{ADM}(g_r) = \frac{R}{12} r^3 + \lt( \frac{1}{72} |\mbox{Rc}|^2 - \frac{5}{432}R^2 + \frac{1}{120}\Delta R \rt)r^5 + O(r^6) \label{small sphere limits}
\end{align}
as $r \rw 0$. Here $R$ is the scalar curvature,  $|\mbox{Rc}|^2$ the square norm of the Ricci curvature, $\Delta$ the Laplacian, and all terms are evaluated at $p$. 
\end{theorem}

From the work of Huisken-Ilmanen \cite{HI}, the Hawking mass furnishes a lower bound for the {\it Bartnik-Bray outer mass}\footnote{It is a modification of the Bartnik mass that requires $\pl M$ to be outer minimizing instead of \eqref{no horizon} in the definition of $\mathcal{PM}[\Omega,\gamma]$.} $m_{out}$ \cite{Bray} (see also \cite[Section 3.3]{Miao}). The small sphere limits of Hawking mass is evaluated in \cite{FST} 
\begin{align*}
m_H(\pl B_r) = \frac{R}{12} r^3 + \lt( -\frac{1}{144}R^2 + \frac{1}{120}\Delta R \rt)r^5 + O(r^6).
\end{align*}
As \eqref{small sphere limits} gives a upper bound, Wiygul obtains the main theorem in \cite{W1}:
\begin{align*}
\lim_{r\rw 0} \frac{m_{out}(\sigma_r, H_r)}{r^3} = \frac{R}{12};
\end{align*}
if the Riemann curvature tensor vanishes at $p$,
\begin{align*}
\lim_{r\rw 0} \frac{m_{out}(\sigma_r, H_r)}{r^5} = \frac{\Delta R}{120}.
\end{align*}
\begin{rmk}
We are interested in \eqref{small sphere limits} rather than the (more difficult) evaluation of the Bartnik mass or the Bartnik-Bray outer mass. Strictly speaking, it is not necessary to mention the Bartnik mass in this paper. For an introduction on the recent development of the Bartnik mass and its modifications, we recommend the excellent survey by Anderson \cite{A2}.
\end{rmk}

Let us describe the idea of Theorem \ref{Wiygul_main}. The existence part was first proved by Anderson \cite{A1}. However, the proof is abstract and not amenable for the mass estimate. Wiygul gives a new proof of existence based on an iteration of linearized static equations so that the mass can be read off from the solutions. The key lies in Lemma 2.21 of \cite{W1} where he recognizes that at the linear level the solution is a harmonic conformal transformations together with a diffeomorphism.  The harmonicity then enables him to relate the ADM mass to the Bartnik data, (2.25) of \cite{W1}. To compute $r^5$ coefficients, he solves the linearized static equations to the next order and carefully keeps track of the contributions of the various components intrepidly in \cite{W2}.

The goal of this paper is to give a simple proof of \eqref{small sphere limits} assuming the existence of a static extension (Anderson's existence result is sufficient for us). The key to our approach is the well-known Komar-type formula for the ADM mass of a static manifold.
\begin{align}\label{mass integral formula}
m_{ADM}(g) = \frac{1}{4\pi} \int_{\pl M} \frac{\pl V}{\pl \nu}
\end{align}
together with the observation that in the equation
\[ \Delta_g V = \Delta_{\pl M}V + \mbox{Hess}_g V(\nu,\nu) + H \frac{\pl V}{\pl \nu} \] the hessian term can be related to the Bartnik data through the static equations and the Gauss equation. This immediately  yields the $r^3$ expansion and the $r^5$ expansion if we assume the Riemann curvature tensor vanishes at $p$. We treat them in Section 3.   

To get the general $r^5$ expansion, we first solve, in Section 4.1, the linearization of the static potential $\dot V$ using the spherical harmonics decomposition and the linearization of the second fundamental form $\dot h_{ab}$ of $\pl M$ (with respect to the static extension) via the Codazzi equation. With these at hand, the rest is direct computation carried out in Section 4.2 and 4.3.

\section{Setup and preliminaries}
Let $p$ be a point in a 3-dimensional Riemannian manifold $(M^{3},\gamma)$. It is well-known that in normal coordinates $x^{i}$ centered at $p$ the metric has the expansion
\begin{align*}\label{normal coordinate}
\gamma_{ij} &= \delta_{ij} - \frac{1}{3} R_{ikjl} x^k x^l - \frac{1}{6} R_{ikjl;m} x^k x^l x^m \\
&\quad + \lt( -\frac{1}{20} R_{ikjl;mn} + \frac{2}{45} R_{ikpl} R_{jmpn} \rt) x^k x^l x^m x^n + O(r^5).
\end{align*}
Here the Riemann curvature tensor and its covariant derivatives on the right-hand side are evaluated at $p$. See \cite[(5.4)]{LP}

We compute the induced metric of the geodesic sphere of radius $r$, $\pl B_r(p)$, 
\begin{align*}
\sigma_{ab} = r^2 \tilde\sigma_{ab} - \frac{1}{3} r^4 R_{ikjl} X^k X^l \pl_a X^i \pl_b X^j + O(r^5)
\end{align*} 
where $x^i = r X^i$ and $\tilde\sigma$ denotes the standard metric on unit sphere. More precisely, we have $X^1 = \sin\theta\cos\phi, X^2 = \sin\theta\sin\phi, X^3 = \cos\theta,$ and $\tilde\sigma = d\theta^2 + \sin^2\theta d\phi^2$. We also compute the second fundamental form
\begin{align*}
h_{ab} &= \frac{1}{2}\pl_r \sigma_{ab}\\
&= r\tilde\sigma_{ab} - \frac{2}{3} r^3 R_{ikjl} X^k X^l \pl_a X^i \pl_b X^j + O(r^4) 
\end{align*}
and the mean curvature
\begin{align*}
H = \frac{2}{r} - \frac{1}{3} r R_{kl} X^kX^l + O(r^2)
\end{align*}
using the identity 
\begin{align}
\tilde\na_a X^i \tilde \na^a X^j = \delta^{ij} - X^i X^j
\end{align}

It is more convenient to rescale the metric on $B_r(p)$ by $r^{-2}$ so that $(B_{r}(p),r^{-2}\gamma)$ has a limit $(B_{1},g_{\text{Euc}})$ as $r \rw 0$. We summarize the above consideration in the following proposition. 
\begin{prop}\label{first order data}
The rescaling of geodesic balls $(B_r(p), r^{-2}\gamma)$ give rise to a family of Bartnik data $(\sigma_r, H_r)$, $0 \le r < \epsilon$, on $S^2$
\begin{align*}
(\sigma_r)_{ab} &= \tilde\sigma_{ab} + r^2 \dot{\sigma}_{ab} + O(r^3)\\
H_r &= 2 + \dot{H} r^2 + O(r^3),
\end{align*}
where
\begin{align*}
\dot\sigma_{ab} &= -\frac{1}{3}R_{ikjl} X^iX^j \pl_a X^k \pl_b X^l\\
\dot H &= -\frac{1}{3}R_{ij}X^iX^j.
\end{align*}
Moreover, the Gauss curvature of $\sigma_r$ is given by
\begin{align*}
K = 1 + \dot K r^2 + O(r^3)
\end{align*}
where
\begin{align}
\dot K &= \frac{1}{2} \lt( - \tilde\sigma^{ab} \dot{\sigma}_{ab} + \tilde\na^a\tilde\na^b \dot{\sigma}_{ab} + \tilde\Delta (\tilde\sigma^{ab} \dot{\sigma}_{ab}) \rt) \notag\\
&= \frac{R}{2} - \frac{4}{3}R_{ij}X^iX^j \label{dot K}
\end{align}
\end{prop}

\begin{lemma}\label{scale}
Let $a>0$ be a constant. Let $(M,g)$ and $(M, \hat g = a^2 g)$ be the extensions of $(\Omega, \gamma)$ and $(\Omega, a^2 \gamma)$ respectively. Then the ADM mass of $g$ and $\hat g$ satisfy
\[ m_{ADM}(\hat g) = a \cdot m_{ADM} (g) \]
\end{lemma}
\begin{proof}
Let $y$ be an asymptotically flat coordinate of $g$. Then $\hat y = a y$ is an asymptotically flat coordinate of $\hat g$. We have
\begin{align*}
\hat g_{ij} = \hat g \lt( \frac{\pl}{\pl \hat y^i}, \frac{\pl}{\pl \hat y^j} \rt) = g_{ij}
\end{align*}
and hence 
\[ \frac{\pl \hat g_{ij}}{\pl \hat y^k} = a^{-1} \frac{\pl g_{ij}}{\pl y^k}. \]
Moreover, we have
\[ d \hat S^i = \frac{\pl}{\pl \hat y^i} \lrcorner \lt( d\hat y^1 \wedge d\hat y^2 \wedge d\hat y^3 \rt) = a^2 d S^i \]
and the assertion follows from the definition of ADM mass \[ m_{ADM}(g) = \frac{1}{16\pi} \lim_{|y|\rw\infty} \int \sum_{j=1}^3 \lt( \frac{\pl g_{ij}}{\pl y^j} - \frac{\pl g_{jj}}{\pl y^i} \rt) dS^i. \] 
\end{proof}

The following theorem is the main result of this paper which, by Lemma \ref{scale}, is equivalent to \eqref{small sphere limits}.
\begin{theorem}\label{main}
Let $g_r$ be the static extension of $(B_r(p), r^{-2}\gamma)$ constructed in Anderson \cite{A1}. Then we have the expansion
\begin{align*}
m_{ADM}(g_r) = \frac{R}{12} r^2 + \lt( \frac{1}{72} |\mbox{Rc}|^2 - \frac{5}{432}R^2 + \frac{1}{120}\Delta R \rt)r^4 + O(r^5).
\end{align*} 
\end{theorem}

Our proof relies on the following well-known formula for ADM mass of static manifolds in terms of the static potential.
\begin{lemma}\label{static mass}
The ADM mass of an asymptotically flat static manifold is equal to $\frac{1}{4\pi} \int_\Sigma \frac{\pl V}{\pl \nu} d\sigma$ where $\Sigma$ is any surface that bounds a domain with a large coordinate sphere.
\end{lemma}
\begin{proof}
(see also Corollary 4.2.4 of \cite{C}) By Remark 2 of \cite{BM}, there exists a coordinate system $y^\alpha$ in which the metric has expansion
\[ \] and the static potential has the expansion
\[ V = 1 - \frac{m}{|y|} + O(|y|^{-2}). \]
In particular, the constant $m$ is equal to the ADM mass. One readily computes that for large coordinate spheres
\begin{align*}
m = \frac{1}{4\pi} \lim_{s \rw \infty} \int_{|y|=s} \frac{\pl V}{\pl\nu} d\sigma.
\end{align*}
Since $V$ is harmonic, we have
\begin{align*}
\frac{1}{4\pi} \int_{\Sigma} \frac{\pl V}{\pl \nu} d\sigma =m
\end{align*}
for any surface $\Sigma$ that bounds a domain with a large coordinate sphere.
\end{proof}

\section{First order computation}
In the rest of this paper, $\int_{S^2} f$ (without mentioning the area form) refers to integrals on $S^2$ with respect to the standard metric $\tilde\sigma$ of unit sphere.
\subsection{}
Let $(M = \R^3 - B_1, g_r, V_r)$ be the static extension of $(\sigma_r, H_r)$ where $g_0$ is the Euclidean metric and $V_0 = 1$. Let $m_r$ be the ADM mass of $g_r$ with Taylor expansion
\[ m_r = \dot m r^2 + \ddot m r^4 + O(r^5). \]

\begin{theorem}
\begin{align*}
\dot m = \frac{R}{12}.
\end{align*}
\end{theorem}
\begin{proof}
Recall the Gauss equation
\begin{align*}
2K &= - 2 Ric(\nu,\nu) + H^2 - |h|^2_\sigma 
\end{align*}
and the identity 
\begin{align*}
0=\Delta_g V = \Delta_{\pl M} V + \mbox{Hess}_g V (\nu,\nu) + H \frac{\pl V}{\pl \nu}
\end{align*}
on static manifolds. The two are related by the static equation $V Ric(\nu,\nu) = \mbox{Hess}_g V(\nu,\nu)$.

We have $V = 1 + \dot{V}r^2 + O(r^3)$ and hence
\begin{align*}
\Delta_{\pl M} V &= \tilde\Delta \dot{V} r^2 + O(r^3)\\
H \frac{\pl V}{\pl\nu} &= 2 \lt( \frac{\pl V}{\pl \nu} \rt)^{\bigcdot} r^2 + O(r^4).
\end{align*}
Moreover, direct computation yields $|h|^2_\sigma = 2 + 2\dot{H}r^2 + O(r^3)$. Comparing the coefficients of $r^2$, we get 
\[ \dot{K} = \tilde\Delta \dot{V} + 2\lt( \frac{\pl V}{\pl \nu} \rt)^{\bigcdot} + \dot{H}. \] and integrate over $\pl M$ to get
\begin{align*}
\dot{m} = \frac{1}{8\pi} \int_{S^2} \dot{K} - \dot{H} = \frac{1}{12} R
\end{align*}
where we used the identity \begin{align}
\int_{S^2} X^i X^j = \frac{4\pi}{3}\delta^{ij}
\end{align} 
\end{proof}

\subsection{Contribution of $\na R$.}
Now we assume that $R_{ijkl}(p)=0$. As term $R_{ikjl;m}$ gives an odd function which integrates to zero and does not contribute to mass, we ignore it. Therefore, we assume
\begin{align*}
(\sigma_r)_{ab} =  \tilde\sigma_{ab} - \frac{r^4}{20}R_{ikjl;mn} X^k X^lX^m X^n \pl_a X^i \pl_b X^j + O(r^5)
\end{align*} 
and
\begin{align*}
H = 2 + r^4 (-\frac{1}{10})R_{kl;mn} X^kX^lX^m X^n + O(r^5). 
\end{align*}

Then the ADM mass has expansion $m = \dot m r^4 + O(r^5)$ and we have\begin{align*}
\dot m &= \frac{1}{8\pi} \int_{S^2} \dot K - \dot H\\
&= \frac{1}{8\pi} \int_{S^2} (-\frac{1}{40} - \frac{1}{10}) R_{kl;mn} X^kX^lX^m X^n \\
&= \frac{1}{8\pi} (-\frac{1}{8}) \frac{4\pi}{15} R_{kl;mn} (\delta^{kl}\delta^{mn} + \delta^{km}\delta^{ln} + \delta^{kn}\delta^{lm})\\
&= -\frac{1}{120}\Delta R
\end{align*}
where we used the identity
\begin{align}
\frac{1}{4\pi}\int_{S^2} X^kX^lX^mX^n = \frac{1}{15}(\delta^{kl}\delta^{mn} + \delta^{km}\delta^{ln} + \delta^{kn}\delta^{lm}) 
\end{align} in the third equality and the second Bianchi identity in the last equality.

\section{Second order computation}
As the contribution of $\na R$ has been handled, we may assume that all covariant derivatives of the Riemann curvature tensor vanish at $p$ in \eqref{normal coordinate}. The results obtained under this assumption can be superposed with that of the previous subsection to yield the desired $r^5$ expansion.

In this section, we raise/lower indices, take trace and norm of tensors on $S^2$ with respect to $\tilde\sigma$. For example, $\tr \dot\sigma = \tilde\sigma^{ab} \dot\sigma_{ab}$ and $|\dot h|^2 = \tilde\sigma^{ab}\tilde\sigma^{cd} \dot h_{ac} \dot h_{bd}$. 
\subsection{Solve first variation data $\dot V$ and $\dot h_{ab}$ explicitly}
\begin{lem}\label{Vdot}
On $\pl M$, we have
\begin{align*}
\dot V = -\frac{R}{12} + \frac{1}{12} R_{ij} \lt( X^iX^j - \frac{1}{3}\delta^{ij} \rt) = -\frac{R}{9} + \frac{1}{12}R_{ij}X^iX^j.
\end{align*}
and
\begin{align*}
\left( \frac{\partial V}{\partial \nu} \right)^{\bigcdot} = \frac{R}{12} - \frac{1}{4} R_{ij} \lt( X^iX^j - \frac{1}{3}\delta^{ij} \rt) = \frac{R}{6}-\frac{1}{4}R_{ij}X^iX^j
\end{align*}
\end{lem}
\begin{proof}
Since $\Delta \dot V =0$, we can expand it into spherical harmonics on $\R^3 - B_1$:
\begin{align*}
\dot V = \sum_{\ell=1}^\infty \frac{\vphi_{\ell-1}}{s^l}
\end{align*}
where $\vphi_\ell$ satisfies $\tilde\Delta \vphi_{\ell} = -\ell(\ell + 1) \vphi_{\ell}$.

On $\pl M$, we have $\tilde\Delta \dot V + 2 \frac{\pl \dot V}{\pl s} = \dot K - \dot H$. Comparing the left-hand side and the right-hand side 
\[ -2 \vphi_0 + \sum_{\ell\ge 2} -\ell(\ell+2) \vphi_\ell  = \frac{R}{2} - R_{ij} X^iX^j = \frac{R}{6} - R_{ij} \lt( X^iX^j - \frac{1}{3}\delta^{ij} \rt), \]
the equation of $\dot V$ follows. The equation of 
$\lt( \frac{\pl V}{\pl \nu} \rt)^{\bigcdot}$ follows from the fact that
\[ \lt( \frac{\pl V}{\pl \nu} \rt)^{\bigcdot} =\lt.  \frac{\pl \dot V}{\pl s} \rt|_{s=1}\]
on $\pl M$.
\end{proof}

\begin{lem}\label{hdot}
\begin{align*}
\dot h_{ab} = -\frac{1}{3}R_{ij}X^iX^j \tilde\sigma_{ab} + \tilde\na_a\tilde\na_b \mathrm{k} - \frac{1}{2}\tilde\Delta \mathrm{k} \tilde\sigma_{ab}
\end{align*}
where 
\begin{align*}
\mathrm{k} = \frac{1}{12} R_{ij}\lt( X^iX^j - \frac{1}{3}\delta^{ij}\rt)
\end{align*}
\end{lem}
\begin{proof}
We decompose $\dot h_{ab}$ into 
\begin{align*}
\dot h_{ab} = \frac{1}{2} \tr \dot h \tilde\sigma_{ab} + k_{ab}
\end{align*}
where $k_{ab}$ is traceless with respect to $\tilde\sigma$. Since $\dot H = \tilde\sigma^{ab}(\dot h_{ab} - \dot\sigma_{ab})$, we have 
\begin{align*}
\tr \dot h = -\frac{2}{3} R_{ij}X^iX^j.
\end{align*}

It is well-known (see Appendix B of \cite{KWY} for example) that $k_{ab}$ admits the decomposition
\begin{align*}
k_{ab} = \tilde\na_a\tilde\na_b \mathrm k - \frac{1}{2} \tilde\Delta \mathrm k \tilde\sigma_{ab} + \mbox{co-closed part}.
\end{align*}
 The Codazzi equation on static manifold reads 
\begin{align*}
\na^a h_{ab} - \pl_b H = Ric (\pl_b, \nu) = \frac{1}{V} \lt( \pl_b \frac{\pl V}{\pl\nu} - h_{bc} \na^c V \rt). 
\end{align*}
On $\pl M$, we have
\begin{align*}
\na^a h_{ab} = r^2  \lt( \tilde\na^a \dot h_{ab} - \tilde\na^a \dot\sigma_{ab} \rt)
\end{align*}
and 
\begin{align*}
\frac{1}{V} \lt( \pl_b \frac{\pl V}{\pl\nu} - h_{bc} \na^c V \rt) = r^2 \lt[ \pl_b  \lt( \frac{\pl V}{\pl\nu}\rt)^{\bigcdot} - \pl_b \dot V \rt].
\end{align*}
We compute $\tilde\na^a \dot\sigma_{ab} = \frac{1}{6} R_{jk} \pl_b \lt( X^jX^k\rt)$. As a result, the co-closed part of $k_{ab}$ vanishes,
\begin{align*}
\frac{1}{2}\pl_b (\tilde\Delta+2) \mathrm k = -\frac{1}{6} R_{ij} \pl_b \lt( X^iX^j - \frac{1}{3}\delta^{ij} \rt),
\end{align*}
and we obtain
\begin{align*}
\mathrm{k} = \frac{1}{12} R_{ij}\lt( X^iX^j - \frac{1}{3}\delta^{ij}\rt).
\end{align*} 
\end{proof}
\begin{lem}\label{DeltaV}
\begin{align*}
\Delta_\sigma V = r^2 \tilde\Delta \dot V + r^4 \lt( \tilde\Delta \ddot V + \dot\Delta\dot V \rt) + O(r^5)
\end{align*}
where
\begin{align*}
\dot\Delta \dot V = \frac{1}{2} (\pl_a \tr\dot\sigma ) \tilde\na^a \dot V - \tilde\na_a \lt( \dot\sigma^{ab} \tilde\na_b \dot V \rt).
\end{align*}
\end{lem}
\begin{proof}
The assertion follows from the expansion of the operator $\Delta_\sigma$ on $\pl M$:
\begin{align*}
\Delta_\sigma &= \frac{1}{\sqrt{\det\sigma}} \pl_a \sqrt{\det\sigma} \sigma^{ab} \pl_ b \\
&= \tilde\Delta + r^2 \lt( -\frac{1}{2} \frac{1}{\sqrt{\det\sigma}} \tr\dot\sigma \pl_a \sqrt{\det\sigma} \tilde\sigma^{ab} \pl_b + \frac{1}{\sqrt{\det\sigma}}\pl_a \lt( \frac{1}{2} \sqrt{\det\sigma} \tr\dot\sigma \tilde\sigma^{ab} - \sqrt{\det\sigma} \dot\sigma^{ab} \rt)\pl_b \rt) + O(r^4)\\
&= \tilde\Delta + r^2 \lt( \frac{1}{2} ( \pl_a \tr\dot\sigma) \tilde\na^a - \tilde\na_a \dot\sigma^{ab} \tilde\na_b \rt) + O(r^4).
\end{align*}
\end{proof}

\subsection{Second order evaluation}
We have
\begin{align*}
\ddot m = \frac{1}{4\pi} \int_{S^2} \lt( \frac{\pl V}{\pl\nu}\rt)^{\bigcdot\bigcdot} + \frac{1}{2} \lt( \frac{\pl V}{\pl\nu}\rt)^{\bigcdot}\tr\dot\sigma.
\end{align*}
It is straightforward to evaluate the second term
\begin{align}\label{2nd in ddot m}
\frac{1}{4\pi} \int_{S^2} \frac{1}{2}\lt( \frac{\pl V}{\pl\nu}\rt)^{\bigcdot}\tr\dot\sigma = -\frac{R^2}{108} + \frac{1}{24} \frac{R^2 + 2|Rc|^2}{15} = -\frac{7}{1080}R^2 + \frac{1}{180}|Rc|^2.
\end{align}
For the first term, we have
\begin{align*}
\frac{1}{4\pi}\int_{S^2} \lt( \frac{\pl V}{\pl\nu}\rt)^{\bigcdot\bigcdot} &= \frac{1}{4\pi}\int_{S^2} -\frac{1}{2}\dot\Delta \dot V - \frac{1}{2} Rc(\nu,\nu)^{\bigcdot\bigcdot} + \frac{1}{2} \dot V (\dot K - \dot H) - \frac{1}{2}\dot H \lt( \frac{\pl V}{\pl\nu}\rt)^{\bigcdot}.
\end{align*}

It is straightforward to evaluate terms other than $Rc(\nu,\nu)^{\bigcdot\bigcdot}$ using Lemma \ref{DeltaV} and \ref{Vdot} and we omit the proof.
\begin{lem}
\begin{align*}
\frac{1}{4\pi} \int_{S^2} -\frac{1}{2}\dot\Delta \dot V = \frac{|Rc|^2}{180} - \frac{R^2}{540}
\end{align*}
\begin{align*}
\frac{1}{4\pi} \int_{S^2} \frac{1}{2}\dot V(\dot K - \dot H) = -\frac{11R^2}{2160} - \frac{|Rc|^2}{180}
\end{align*}
\begin{align*}
\frac{1}{4\pi} \int_{S^2} - \frac{1}{2}\dot H \lt( \frac{\pl V}{\pl\nu}\rt)^{\bigcdot} = \frac{7R^2}{1080} - \frac{|Rc|^2}{180}
\end{align*}
\end{lem}

Putting these together, we obtain
\begin{align}\label{ddot m, temp}
\ddot m = -\frac{R^2}{144} + \frac{1}{4\pi} \int_{S^2} -\frac{1}{2} Ric(\nu,\nu)^{\bigcdot\bigcdot}
\end{align}
\subsubsection{Evaluation of $\frac{1}{4\pi}\int_{S^2} -\frac{1}{2} Ric(\nu,\nu)^{\bigcdot\bigcdot}$}
\begin{lem}
\begin{align*}
\ddot H = -\frac{1}{45} R_{ikjl}R_{imjn} X^kX^lX^mX^n.
\end{align*}
\end{lem}
\begin{proof}
For small spheres, we have $H = \lt( \sigma^{-1}\rt)^{ab} \frac{1}{2}\pl_r \sigma_{ab}$ with
\begin{align*}
\lt( \sigma^{-1} \rt)^{ab} &=  \frac{\tilde\sigma^{ab}}{r^2} + \frac{1}{3}R_{ikjl}X^kX^l \tilde\na^a X^i \tilde\na^b X^j \\
&\quad + r^2 \lt( -\frac{2}{45}+\frac{1}{9}\rt) \sum_k R_{imkn}R_{jpkq} \tilde\na^a X^i \tilde \na^b X^j X^mX^nX^pX^q + O(r^3)
\end{align*}
and
\begin{align*}
\frac{1}{2} \pl_r \sigma_{ab} &= r\tilde\sigma_{ab} - \frac{2}{3}r^3 R_{mpnq} X^pX^q\pl_a X^m\pl_b X^n \\
&\quad + \frac{6}{45}r^5  \sum_k R_{imkn}R_{jpkq} \pl_a X^i \pl_b X^j X^mX^pX^nX^q. 
\end{align*}
Hence we get the term $-\frac{1}{45}r^5 R_{ikjl}R_{imjn} X^kX^lX^mX^n$ and the assertion then follows from scaling.
\end{proof}
We first compute
\begin{align*}
\lt( |h|^2_\sigma \rt)^{\bigcdot\bigcdot} = 2\ddot H + |\dot\sigma|^2 + |\dot h|^2 - 2 \dot\sigma^{ab} \dot h_{ab}
\end{align*}
and obtain, via Gauss equation, 
\begin{align*}
\int_{S^2} - \frac{1}{2} Rc(\nu,\nu)^{\bigcdot\bigcdot} &= \int_{S^2} \frac{1}{2} \lt( \ddot K - \frac{1}{2}\dot H^2 -2 \ddot H + \frac{1}{2} \lt( |h|^2 \rt)^{\bigcdot\bigcdot} \rt)\\
&= \int_{S^2} \frac{1}{2}\ddot K - \frac{1}{4} \dot H^2 -\frac{1}{2}\ddot H + \frac{1}{4}|\dot\sigma|^2 + \frac{1}{4}|\dot h|^2 - \frac{1}{2}\dot\sigma^{ab}\dot h_{ab}\\
&= I_1 + I_2 + I_3 + I_4 + I_5 + I_6. 
\end{align*}

We introduce two quantities \begin{align*}
A &= R_{minj}R^{mknl} \frac{\delta^{ij}\delta^{kl} + \delta^{ik}\delta^{jl} + \delta^{il}\delta^{jk}}{15} = \frac{1}{15}\lt( |Rc|^2 + |Rm|^2 + R_{minj}R^{mjni} \rt)\\
&= \frac{1}{15} \lt( 7 |Rc|^2 - \frac{3}{2}R^2 \rt)\\
B &= R_{ij}R_{kl} \frac{\delta^{ij}\delta^{kl} + \delta^{ik}\delta^{jl} + \delta^{il}\delta^{jk}}{15} = \frac{1}{15} \lt( R^2 + 2|Rc|^2 \rt) 
\end{align*}
where we use (A.5) of \cite{W2} in the last equality of $A$.

Terms directly related to Bartnik data can be readily evaluated and we omit the proof. 
\begin{lem}
\begin{align*}
I_2 = \frac{1}{4\pi}\int_{S^2} -\frac{1}{4} \dot H^2 = -\frac{1}{36}B = -\frac{R^2}{540} - \frac{|Rc|^2}{270}
\end{align*}

\begin{align*}
I_3 = \frac{1}{4\pi}\int_{S^2} -\frac{1}{2}\ddot H = \frac{1}{90} A
\end{align*}

\begin{align*}
I_4 = \frac{1}{4\pi}\int_{S^2} \frac{1}{4}|\dot\sigma|^2 = \frac{1}{36} A
\end{align*}
\end{lem}
The next two lemma handle terms coming from $\dot h$.
\begin{lem}
\begin{align*}
I_5 = \frac{1}{4\pi}\int_{S^2} \frac{1}{4}|\dot h|^2 = \frac{11}{144} B - \frac{R^2}{432}. 
\end{align*}
\end{lem}
\begin{proof}
By virtue of Lemma 2.6 of \cite{CKWWY}, we have
\begin{align*}
\frac{1}{4\pi} \int_{S^2} \frac{1}{4} |\dot h|^2 &= \frac{1}{4\pi} \int_{S^2} \frac{1}{18} R_{ij}X^iX^j R_{kl}X^kX^l + \frac{1}{4} |k|^2 \\
&= \frac{1}{18}B + \frac{1}{4\pi} \int_{S^2} \frac{1}{8}\cdot \frac{1}{6} R_{ij} (X^iX^j - \frac{1}{3}\delta^{ij}) R_{kl} (X^kX^l - \frac{1}{3}\delta^{kl})\\ &= \frac{1}{18}B + \frac{1}{48}B - \frac{R^2}{432}
\end{align*}
\end{proof}

\begin{lem}
\begin{align*}
I_6 = \frac{1}{4\pi}\int_{S^2} -\frac{1}{2}\dot\sigma^{ab} \dot h_{ab} = -\frac{1}{24}B + \frac{|Rc|^2}{108} -\frac{R^2}{216}
\end{align*}
\end{lem}
\begin{proof}
We compute \begin{align*}
\na_a \na_b \mathrm{k} - \frac{1}{2}\tilde\Delta \mathrm k \tilde\sigma_{ab} = \frac{1}{12}R_{ij} \lt( \na_a X^i \na_b X^j + \na_b X^i \na_a X^j - (\delta^{ij}-X^iX^j)\tilde\sigma_{ab} \rt)
\end{align*}
to get
\begin{align*}
\dot\sigma^{ab} \dot h_{ab} &= \frac{1}{9}R_{kl} X^kX^l R_{mn}X^mX^n\\
&\quad - \frac{1}{36}R_{ikjl}X^kX^l R_{mn}\lt( \delta^{im}\delta^{jn} + \delta^{in}\delta^{jm}-(\delta^{ij}-X^iX^j)(\delta^{mn}-X^mX^n) \rt)\\
&= \frac{1}{12} R_{kl} X^kX^l R_{mn}X^mX^n - \frac{1}{18}R_{ikjl}X^kX^l R^{ij} + \frac{1}{36}R_{kl}X^kX^l R.
\end{align*}
\end{proof}
Lastly, we evaluate $I_1$. Taking the second variation of the Gauss-Bonnet Theorem, we get
\begin{align*}
\frac{1}{4\pi}\int_{S^2} \ddot K &= \frac{1}{4\pi}\int_{S^2} - \frac{1}{4\pi}\frac{1}{2}\dot K \tr\dot\sigma + \int_{S^2} -\frac{1}{2}\tr\ddot\sigma - \frac{1}{8} (\tr\dot\sigma)^2 + \frac{1}{4}|\dot\sigma|^2 \\
&= \lt( -\frac{2}{9}B + \frac{R^2}{36} \rt) + \lt( -\frac{1}{45}A - \frac{1}{72}B + \frac{1}{36}A \rt)\\
&= \frac{R^2}{36} + \frac{1}{180}A - \frac{17}{72}B
\end{align*}

\begin{lem}
\begin{align*}
\frac{1}{4\pi}\int_{S^2} -\frac{1}{2}Rc(\nu,\nu)^{\bigcdot\bigcdot} = \frac{1}{72}|Rc|^2 - \frac{R^2}{216}.
\end{align*}
\end{lem}
\begin{proof}
We add up the results from previous lemmas
\begin{align*}\def\arraystretch{1.2}
\begin{array}{ccccc}
& R^2 & A & B & |Rc|^2\\
\hline
I_1 & \frac{1}{72} & \frac{1}{360} & -\frac{17}{144} \\
I_2 & -\frac{1}{540} & & & -\frac{1}{270}\\
I_3 & & \frac{1}{90}\\
I_4 & & \frac{1}{36}\\
I_5 & -\frac{1}{432} & & \frac{11}{144}\\
I_6 & -\frac{1}{216} & & -\frac{1}{24}1 & \frac{1}{108}
\end{array}
\end{align*}
to get
\begin{align*}
\frac{1}{4\pi}\int_{S^2} -\frac{1}{2}Rc(\nu,\nu)^{\bigcdot\bigcdot} &= \frac{11}{2160}R^2 + \frac{1}{24}A - \frac{1}{12}B + \frac{1}{180}|Rc|^2 \\
&= -\frac{1}{216}R^2 + \frac{1}{72}|Rc|^2
\end{align*}
\end{proof}
Recalling \eqref{ddot m, temp}, we obtain
\[ \ddot m = -\frac{R^2}{144} + \frac{1}{4\pi}\int_{S^2} -\frac{1}{2} Rc(\nu,\nu)^{\bigcdot\bigcdot} = \frac{1}{72}|Rc|^2 - \frac{25}{2160} R^2. \]
This completes the proof of Theorem \ref{main}.


\begin{thebibliography}{Besse}
\bibitem{A1} Anderson, Michael T. Local existence and uniqueness for exterior static vacuum Einstein metrics. Proc. Amer. Math. Soc. 143 (2015), no. 7, 3091–3096.
\bibitem{A2}  Anderson, Michael T. {\it Recent progress and problems on the Bartnik quasi-local mass.} Pure Appl. Math. Q. 15 (2019), no. 3, 851–873.
\bibitem{B} Bartnik, Robert {\it New definition of quasilocal mass.} Phys. Rev. Lett. 62 (1989), no. 20, 2346–2348.
\bibitem{Bray} Bray, Hubert L.; Chruściel, Piotr T. {\it The Penrose inequality.} The Einstein equations and the large scale behavior of gravitational fields, 39–70, Birkhäuser, Basel, 2004. 
\bibitem{BM} Bunting, Gary L.; Masood-ul-Alam, A. K. M. {\it Nonexistence of multiple black holes in asymptotically Euclidean static vacuum space-time.} Gen. Relativity Gravitation 19 (1987), no. 2, 147–154.
\bibitem{C} Cederbaum, Carla {\it The Newtonian Limit of Geometrostatics.} Dissertation available at https://arxiv.org/abs/1201.5433
\bibitem{CKWWY}  Chen, Po-Ning; Keller, Jordan; Wang, Mu-Tao; Wang, Ye-Kai; Yau, Shing-Tung {\it Evolution of angular momentum and center of mass at null infinity.} Comm. Math. Phys. 386 (2021), no. 1, 551--588.
\bibitem{FST} Fan, Xu-Qian; Shi, Yuguang; Tam, Luen-Fai {\it Large-sphere and small-sphere limits of the Brown-York mass.} Comm. Anal. Geom. 17 (2009), no. 1, 37–72.
\bibitem{HI} Huisken, Gerhard; Ilmanen, Tom {\it The inverse mean curvature flow and the Riemannian Penrose inequality.} J. Differential Geom. 59 (2001), no. 3, 353–437.
\bibitem{KWY} Keller, Jordan; Wang, Ye-Kai; Yau, Shing-Tung {\it Evaluating quasi-local angular momentum and center-of-mass at null infinity.} Adv. Theor. Math. Phys. 24 (2020), no. 6, 1423--1473.
\bibitem{LP}  Lee, John M.; Parker, Thomas H. {\it The Yamabe problem.} Bull. Amer. Math. Soc. (N.S.) 17 (1987), no. 1, 37–91.
\bibitem{Miao} Miao, Pengzi {\it Variational effect of boundary mean curvature on ADM mass in general relativity.} Mathematical physics research on the leading edge, 145–171, Nova Sci. Publ., Hauppauge, NY, 2004.
\bibitem{W1}  Wiygul, David {\it The Bartnik-Bray outer mass of small metric spheres in time-symmetric 3-slices.} Comm. Math. Phys. 358 (2018), no. 1, 269–293.
\bibitem{W2} Wiygul, David {\it Second-order mass estimates for static vacuum metrics with small Bartnik data.} arXiv:2110.12771 
\end{thebibliography}
\end{document}